\documentclass[11pt,a4paper,fullpage]{article}
\usepackage{amsmath,amsthm}
\usepackage{amssymb}

\newtheorem{theorem}{Theorem}
\newtheorem{proposition}[theorem]{Proposition}
\newtheorem{lemma}[theorem]{Lemma}
\newtheorem{Remark}[theorem]{Remarks}
\newtheorem{corollary}[theorem]{Corollary}
\newtheorem{Conjecture}[theorem]{Conjecture}

\title{{\bf Dedekind sums and mean square value of $L(1,\chi)$ over subgroups}}
\author{
St\'ephane R. LOUBOUTIN\\
Aix Marseille Universit\'e, CNRS, Centrale Marseille, I2M,\\ 
Marseille, France\\
stephane.louboutin@univ-amu.fr}
\date{}

\begin{document}
\bibliographystyle{alpha}
\maketitle

\centerline{To appear in Publ. Math. Debrecen}
\footnotetext{

Mean square value. 
Relative class number. 
Dedekind sums. Cyclotomic field.}

\begin{abstract}
An explicit formula for the quadratic mean value at $s=1$ of the Dirichlet $L$-functions 
associated with the odd Dirichlet characters modulo $f>2$ is known. 
Here we present a situation where we could prove an explicit formula 
for the quadratic mean value at $s=1$ of the Dirichlet $L$-functions 
associated with the odd Dirichlet characters modulo not necessarily prime moduli $f>2$ 
that are trivial on a subgroup $H$ of the multiplicative group $({\mathbb Z}/f{\mathbb Z})^*$. 
This explicit formula involves summation $S(H,f)$ of Dedekind sums $s(h,f)$ over the $h\in H$. 
A result on some cancelation of the denominators of the $s(h,f)$'s when computing $S(H,f)$ is known. 
Here, we prove that for some explicit families of $f$'s and $H$'s 
this known result on cancelation of denominators is the best result one can expect. 
Finally, we surprisingly prove that for $p$ a prime, $m\geq 2$ and $1\leq n\leq m/2$, 
the values of the Dedekind sums $s(h,p^m)$ 
do not depend on $h$ 
as $h$ runs over the elements of order $p^n$ 
of the multiplicative cyclic group $({\mathbb Z}/p^m{\mathbb Z})^*$.
\end{abstract}

\section{A general mean square value formula in terms of Dedekind sums}
For $c\in {\mathbb Z}$ and $d>1$ such that $\gcd (c,d)=1$,
the {\it Dedekind sum} is defined by
\begin{equation}\label{defs}
s(c,d)
:=\frac{1}{4d}\sum_{a=1}^{d -1}\cot\left (\frac{\pi a}{d}\right )\cot\left (\frac{\pi ac}{d}\right )
\end{equation}
(see \cite[Chapter 3, Exercise 11]{Apo} or \cite[(26)]{RG}).
It depends only on $c$ modulo $d$.
We also set $s(c,1)=0$ for $c\in {\mathbb Z}$.
Notice that $s(c^*,d)=s(c,d)$ whenever $cc^*\equiv 1\pmod d$ 
(make the change of variables $n\mapsto nc$ in $s(c^*,d)$). 
We have a reciprocity law for Dedekind sums 
(see e.g. \cite[Theorem 3.7]{Apo}, \cite[(4)]{RG}) or \cite[(7) and (9)]{LouBKMS52})
$$s(c,d)+s(d,c) 
=\frac{c^2+d^2-3cd+1}{12cd}
\ \ \ \ \ (c>1,\ d>1,\ \gcd(c,d)=1).$$
We deduce (by induction) that $s(c,d)\in {\mathbb Q}$ 
and that (see also \cite[Lemma (a)(i)]{LouCMB36/37})
\begin{equation}\label{s1d}
s(1,d)
=\frac{(d-1)(d -2)}{12d}
\ \ \ \ \ (d\geq 1).
\end{equation}
For $d>1$ and $\gcd(c,d)=1$, we set 
\begin{equation}\label{deftildes}
\tilde s(c,d)
:=\frac{1}{4d}\sum_{\substack{n=1\\\gcd (n,d)=1}}^{d-1}
\cot\left (\frac{\pi n}{d}\right )\cot\left (\frac{\pi nc}{d}\right ).
\end{equation} 
Using \eqref{defs} we have
\begin{equation}\label{tildescddc}
\tilde s(c,d)
=\sum_{\delta\mid d}\frac{\mu (\delta)}{\delta}s(c,d/\delta).
\end{equation}
In particular, using \eqref{s1d} we obtain
\begin{equation}\label{stilde(1,d)} 
\tilde s(1,d) 
=\frac{\phi(d)}{12}\left (\prod_{p\mid d}\left (1+\frac{1}{p}\right )-\frac{3}{d}\right )
\ \ \ \ \ (d>1).
\end{equation}

For $f>2$, let $X_f$ be the group of order $\phi( f)$ of the Dirichlet characters modulo $f$. 
Let 
$X_f^-:=\{\chi\in X_f\hbox{ and }\chi (-1)=-1\}$ 
be the set of the $\phi (f)/2$ odd Dirichlet characters modulo $f$. 
If $H$ is a subgroup of order $n$ of the multiplicative group $({\mathbb Z}/f{\mathbb Z})^*$ 
which does not contain $-1$, 
we set 
$$X_f^-(H)=\{\chi\in X_f^-;\ \chi_{/H}=1\}.$$
Hence, $\# X_f^-(H) =\phi (f)/(2n)$.
Let $L(s,\chi)$ 
be the Dirichlet $L$-function associated with $\chi\in X_f$. 

\noindent\frame{\vbox{
\begin{theorem}
Let $H$ be a subgroup of order $n$ of the multiplicative group $({\mathbb Z}/f{\mathbb Z})^*$, 
with $f>2$. 
Assume that $-1\not\in H$, which is the case if $n$ is odd.
We have the mean square value formula
\begin{equation}\label{firstformula}
M(f,H)
:=\frac{1}{\# X_f^-(H)}\sum_{\chi\in X_f^-(H)}\vert L(1,\chi)\vert^2
=\frac{2\pi^2}{f}\tilde S(H,f),
\hbox{ where }
\tilde S(H,f)
:=\sum_{h\in H}\tilde s(h,f).
\end{equation}
In particular, by \eqref{stilde(1,d)}, 
we have the mean square value formula (see also \cite[Theorem 2]{LouBPASM64})
\begin{equation}\label{Mf1}
M(f,\{1\})
:=\frac{2}{\phi (f)}\sum_{\chi\in X_f^-}\vert L(1,\chi)\vert^2
=\tilde s(1,f)
=\frac{\pi^2}{6}\frac{\phi (f)}{f}\left (
\prod_{p\mid f} \left (1+\frac{1}{p}\right ) -\frac{3}{f}
\right ).
\end{equation} 
\end{theorem}
}}

\begin{proof}
For \eqref{firstformula}, see \cite[Proof of Theorem 2]{LouBPASM64}.
\end{proof}

\noindent\frame{\vbox{
\begin{corollary}\label{MpHn}
Let $n\geq 1$ be an odd divisor of $p-1$, 
where $p\geq 3$ is an odd prime number. 
Let $H_n$ be the only subgroup of order $n$ 
of the multiplicative cyclic group $({\mathbb Z}/p{\mathbb Z})^*$.
We have the mean square value formula
$$M(p,H_n)
:=\frac{1}{\# X_p^-(H_n) }\sum_{\chi\in X_p^-(H_n)}\vert L(1,\chi)\vert^2
=\frac{2\pi^2}{p}S(H_n,p)
=\frac{\pi^2}{6}
\left (
1
+\frac{N(H_n,p)}{p}
\right ),$$
where 
$$S(H_n,p)
:=\sum_{h\in H_n}s(h,p)
\hbox{ and }
N(H_n,p)
:=12S(H_n,p) -p.$$
Moreover, by \cite[Theorem 6]{LouBKMS56}, for $n>1$ 
the rational number $2S(H_n,p)$ is an integer of the same parity as $(p-1)/2$ 
and $N(H_n,p) =12S(H_n,p)-p$ is an odd rational integer.
\end{corollary}
}}

By \cite[Theorem 1.1]{LMQJM}, we have 
$$N(p,H_n) =o(p)
\text{ and }
M(p,H_n) =\frac{\pi^2}{6}+o(1)$$
as $p$ tends to infinity 
and $H_n$ runs over the subgroup of $({\mathbb Z}/p{\mathbb Z})^*$ 
of odd orders 
$n\leq\frac{\log p}{3\log\log p}$. 
By \cite[Theorem 2.1 and Remark 2.2]{MS} we can relax this constraint on $n$ 
down to 
$\phi(n)=o(\log p)$, which is optimal.
Indeed, if $p$ runs over the Mersenne primes $p=2^n-1$, $n\geq 3$ odd and prime, 
and $H_n$ is the subgroup of order $n$ of $({\mathbb Z}/p{\mathbb Z})^*$ generated by $2$, 
then 
$N(p,H_n)
=2p-(6n-3)$, 
by \cite[Theorem 5.4]{LMCJM}.

\newpage
\section{A conjecture for the case of prime moduli}
According to our numerical computations it seems reasonable to conjecture the following:

\begin{Conjecture}\label{SpHnmod2}
(i). (See \cite[Section 2.2]{LouBKMS56} for some numerical evidence). 
Let $p$ range over the odd prime integers
 and $H$ over the subgroups of odd order of the multiplicative cyclic group $({\mathbb Z}/p{\mathbb Z})^*$. 
Then $N(H,p) =12S(H,p)-p\leq 0$ and hence $M(p,H)\leq\pi^2/6$
with a probability greater than or equal to $1/2$, i.e. 
$$\liminf_{B\rightarrow\infty} \rho(B)
\geq\frac{1}{2},
\hbox{ where }
\rho(B)
:=\frac{\#\{(p,n);\hbox{ $n\geq 1$ odd divides $p-1$, $N(H_n,p)\leq 0$ and $p\leq B$}\}}
{\#\{(p,n);\hbox{ $n\geq 1$ odd divides $p-1$ and $p\leq B$}\}}.$$
(ii). For a given odd integer $n\geq 3$, let $p$ range over the odd prime integers $p\equiv 1\pmod{2n}$. 
Let $H_n$ be the only subgroup of order $n$ 
of the multiplicative cyclic group $({\mathbb Z}/p{\mathbb Z})^*$. 
Then $N(H_n,p) =12S(H_n,p)-p\leq 0$ and hence $M(p,H_n)\leq\pi^2/6$
with a probability greater than or equal to $1/2$, i.e. 
$$\liminf_{B\rightarrow\infty} \rho_n(B)\geq\frac{1}{2},
\hbox{ where }
 \rho_n(B)
 :=\frac{\#\{p;\hbox{ $p\equiv 1\pmod {2n}$, $N(H_n,p)\leq 0$ and $p\leq B$}\}}
 {\#\{p;\hbox{ $p\equiv 1\pmod {2n}$ and $p\leq B$}\}}.$$
\end{Conjecture}

If $p\equiv 1\pmod 6$ then $N(H_3,p) =-1$ (Corollary \ref{thp3}), 
and point (ii) of Conjecture \ref{SpHnmod2} holds true for $n=3$. 
As an example of our computation, for $n=9$ we have the following numerical datas:
{\small
$$\begin{array}{|c|r|r|c|}
\hline
B&\#\{p\leq B;\hbox{ $p\equiv 1\pmod {18}$}\}
&\#\{p;\hbox{ $p\leq B\equiv 1\pmod {18}$ and $N(H_9,p)\leq 0$}\}
&\rho_9(B)\\
\hline
10^5&1592&838&0.52638\cdots\\
10^6&13 063&6 820&0.52208\cdots\\
10^7&110 772&56 779&0.51257\cdots\\
10^8&959 959&490 984&0.51146\cdots\\
10^9&8 474 566&4 317 341&0.50944\cdots\\
10^{10}&75 841 588&38 573 928&0.50861\cdots\\
10^{11}&686 345 266&348 497 259&0.50775\cdots\\
\hline
\end{array}$$
}
and
{\small
$$\begin{array}{|c|c|r|r|c|}
\hline
A&B&c_{prime}(A,B)&c_{\leq 0}(A,B)&\rho_9(A,B)\\
\hline
10^{10}&10^6&7 226&3 695&0.51134\cdots\\
10^{10}&10^7&72 505&36 731&0.50659\cdots\\
10^{10}&10^8&724 408&368 910&0.50925\cdots\\
10^{10}&10^9&7 224 235&3 672 183&0.50831\cdots\\
10^{10}&10^{10}&71 191 018&36 166 905&0.50802\cdots\\
\hline
10^{11}&10^6&6 558&3 301&0.50335\cdots\\
10^{11}&10^7&65 747&33 253&0.50577\cdots\\
10^{11}&10^8&658 053&333 640&0.50701\cdots\\
10^{11}&10^9&6 579 598&3 337 952&0.50731\cdots\\
10^{11}&10^{10}&65 673 261&33 320 115&0.50736\cdots\\
\hline
10^{12}&10^6&6 076&3 145&0.51761\cdots\\
10^{12}&10^7&60 361&30 850&0.51109\cdots\\
10^{12}&10^8&602 908&305 658&0.50697\cdots\\
10^{12}&10^9&6 031 209&3 056 473&0.50677\cdots\\
10^{12}&10^{10}&60 305 132&30 562 355&0.50679\cdots\\
\hline
10^{13}&10^6&5 564&2 782&0.50000\cdots\\
10^{13}&10^7&55 572&28 003&0.50390\cdots\\
10^{13}&10^8&557 166&282 186&0.50646\cdots\\
10^{13}&10^9&5 566 301&2 817 547&0.50617\cdots\\
10^{13}&10^{10}&55 673 215&28 179 022&0.50615\cdots\\
\hline
\end{array}$$
where 

$c_{prime}(A,B)
:=\#\{\hbox{$p\equiv 1\pmod {18};\ A\leq p\leq A+B$}\}$, 

$c_{\leq 0}(A,B)
:=\#\{\hbox{$p\equiv 1\pmod {18};\ A\leq p\leq A+B$ and $N(H_9,p)\leq 0$}\}$ 

and 
$\rho_9(A,B)
:=c_{\leq 0}(A,B)/c_{prime}(A,B)$.
}

For $n=5,7,11,13$ and $15$ we have the following numerical datas:
{\small
$$\begin{array}{|c|r|r|c|}
\hline
B
&\#\{p\leq B;\hbox{ $p\equiv 1\pmod {10}$}\}
&\#\{p\leq B;\hbox{ $p\equiv 1\pmod {10}$ and $N(H_5,p)\leq 0$}\}
&\rho_5(B)\\
\hline
10^5&2387&1335&0.55927\cdots\\
10^6&19 617&10 403&0.53030\cdots\\
10^7&166 104&86 814&0.52264\cdots\\
10^8&1 440 298&744 791&0.51710\cdots\\
10^9&12 711 386&6 540 511&0.51453\cdots\\
10^{10}&113 761 519&58 352 843&0.51294\cdots\\
\hline
\hline
B
&\#\{p\leq B;\hbox{ $p\equiv 1\pmod {14}$}\}
&\#\{p;\hbox{ $p\leq B\equiv 1\pmod {14}$ and $N(H_7,p)\leq 0$}\}
&\rho_7(B)\\
\hline
10^5&1593&823&0.51663\cdots\\
10^6&13 063&6770&0.51825\cdots\\
10^7&110 653&56 848&0.51375\cdots\\
10^8&960 023&490 970&0.51141\cdots\\
10^9&8 474 221&4 322 243&0.51004\cdots\\
10^{10}&75 840 762&38 584 999&0.50876\cdots\\
\hline
\hline
B
&\#\{p\leq B;\hbox{ $p\equiv 1\pmod {22}$}\}
&\#\{p\leq B;\hbox{ $p\equiv 1\pmod {22}$ and $N(H_{11},p)\leq 0$}\}
&\rho_{11}(B)\\
\hline
10^5&945&506&0.53544\cdots\\
10^6&7858&4099&0.52163\cdots\\
10^7&66 386&34 669&0.52223\cdots\\
10^8&576 103&300 012&0.52076\cdots\\
10^9&5 084 435&2 634 688&0.51818\cdots\\
10^{10}&45 504 543&23 481 241&0.51601\cdots\\
\hline
\hline
B
&\#\{p\leq B;\hbox{ $p\equiv 1\pmod {26}$}\}
&\#\{p\leq B;\hbox{ $p\equiv 1\pmod {26}$ and $N(H_{13},p)\leq 0$}\}
&\rho_{13}(B)\\
\hline
10^5&798&397&0.49749\cdots\\
10^6&6539&3307&0.50573\cdots\\
10^7&55 376&28 071&0.50691\cdots\\
10^8&480 132&242 633&0.50534\cdots\\
10^9&4 237 228&2 139 817&0.50500\cdots\\
10^{10}&37 919 477&19 125 424&0.50436\cdots\\
\hline
\hline
B
&\#\{p\leq B;\hbox{ $p\equiv 1\pmod {30}$}\}
&\#\{p\leq B;\hbox{ $p\equiv 1\pmod {30}$ and $N(H_{15},p)\leq 0$}\}
&\rho_{15}(B)\\
\hline
10^5&1189&648&0.54499\cdots\\
10^6&9807&5129&0.52299\cdots\\
10^7&83 003&42 787&0.51548\cdots\\
10^8&719 984&368 612&0.51197\cdots\\
10^9&6 355 189&3 240 295&0.50986\cdots\\
10^{10}&56 878 661&28 940 619&0.50881\cdots\\
\hline
\end{array}$$
}

\section{Mean square values of $L(1,\chi)$ over subgroups 
and bounds on relative class numbers of imaginary abelian number fields}
We refer the reader to \cite[Chapters 3, 4 and 11]{Was} for more background details. 
Let $K$ be an imaginary abelian number field of degree $m =2n>1$ and conductor $f_K>1$. 
Let $f>1$ be any integer divisible by $f_K$, 
i.e. let $K$ be an imaginary subfield of a cyclotomic number field ${\mathbb Q}(\zeta_f)$ 
(Kronecker-Weber's theorem). 
Let $w_K$ be its number of complex roots of unity. 
Let $Q_K\in\{1,2\}$ be its Hasse unit index. 
Hence, $Q_K=1$ if $K/{\mathbb Q}$ is cyclic 
(see e.g. \cite[Example 5, page 352]{Lem}). 
In particular, for any imaginary subfield $K$ of ${\mathbb Q}(\zeta_p)$ 
we have $Q_K=1$ and $w_K=2p$ if $K={\mathbb Q}(\zeta_p)$ 
but $w_K=2$ if $K\varsubsetneq {\mathbb Q}(\zeta_p)$ (see \cite[Exercise 2.3]{Was}).
Let $K^+$ be the maximal real subfield of $K$ of degree $n$ fixed by the complex conjugation. 
The class number $h_{K^+}$ of $K^+$ divides the class number $h_K$ of $K$. 
The {\it relative class number} of $K$ is defined by $h_K^- =h_K/h_{K^+}$. 
Let $d_K$ and $d_{K^+}$ be the absolute values of the discriminants of $K$ and $K^+$. 
For $\gcd (t,f)=1$, 
let $\sigma_t$ be the ${\mathbb Q}$-automorphism of ${\mathbb Q}(\zeta_f)$ 
defined by $\sigma_t(\zeta_f) =\zeta_f^t$. 
Then $t\mapsto\sigma_t$ a is canonical isomorphic 
from the multiplicative group $({\mathbb Z}/f{\mathbb Z})^*$ 
to the Galois group ${\rm Gal}({\mathbb Q}(\zeta_f)/{\mathbb Q})$. 
Set
$$H
:={\rm Gal}({\mathbb Q}(\zeta_f)/K) 
=\{t\in ({\mathbb Z}/f{\mathbb Z})^*;\ \alpha\in K\Rightarrow\sigma_t(\alpha)=\alpha\},$$ 
a subgroup of $({\mathbb Z}/f{\mathbb Z})^*$ of index $m$ and order $\phi (f)/m$. 
Notice also that $\# X_f^-(H) =n$.
Now, $-1\not\in H$ 
(notice that $\sigma_{-1}$ is the complex conjugation restricted to ${\mathbb Q}(\zeta_f)$). 
Any $\chi\in X_f$ is induced by a unique primitive Dirichlet character $\chi^*$ 
of conductor $f_{\chi^*}$ dividing $f$. 
We have the relative class number formula
\begin{equation}\label{formulahrelK}
h_K^-
=\frac{Q_Kw_K}{(2\pi)^n}\sqrt\frac{d_K}{d_{K^+}}\prod_{\chi\in X_f^-(H)} L(1,\chi^*).
\end{equation}
By \cite{LouPJAcad75}, dealing with primitive characters is not going to give explicit formulas.
However, noticing that 
$$L(1,\chi^*) 
=L(1,\chi)\prod_{q\mid f}\left (1-\frac{\chi^* (q)}{q}\right )^{-1}
\ \ \ \ \ (\chi\in X_f)$$ 
and using \eqref{formulahrelK} and the arithmetic-geometric mean inequality, 
we obtain 
\begin{equation}\label{hrelK}
h_K^-
\leq\frac{Q_Kw_K}{\Pi (f,H)}\sqrt\frac{d_K}{d_{K^+}}
\left (\frac{M(f,H)}{4\pi^2}\right )^{n/2},
\end{equation}
where
$$\Pi (f,H)
:=\prod_{q\mid f}
\prod_{\chi\in X_f^-(H)}\left (1-\frac{\chi^*(q)}{q}\right )
\ \ \ \ \ \hbox{($q$ runs over the prime divisors of $f$).}$$
Notice that $\Pi (f,H)=1$ whenever $f=p^m$ is power of a prime. 

For example, let $p\geq 3$ be an odd prime. 
Let $K$ be an imaginary subfield of degree $ (K:{\mathbb Q}) =m$ 
of the cyclotomic field ${\mathbb Q}(\zeta_p)$. 
Set $H ={\rm Gal}({\mathbb Q}(\zeta_p)/K)$, 
a subgroup of order $(p-1)/m$ of the multiplicative group $({\mathbb Z}/p{\mathbb Z})^*$. 
Then $d_K =p^{m-1}$ and $d_{K^+} =p^{m/2-1}$, 
by the conductor-discriminant formula. 
Therefore, by \eqref{hrelK} we have
\begin{equation}\label{hrelKbis}
h_K^-
\leq w_K\left (\frac{pM(p,H)}{4\pi^2}\right )^{m/4},
\hbox{ where }
m= (K:{\mathbb Q})
\hbox{ and }
w_K
=\begin{cases}
2&\hbox{if $K\subsetneq {\mathbb Q}(\zeta_p)$,}\\
2p&\hbox{if $K={\mathbb Q}(\zeta_p)$.}\\
\end{cases}
\end{equation}
In particular, for $H=\{1\}$ and using \eqref{Mf1} and \eqref{hrelKbis} we recover \cite{W}:
\begin{equation}\label{Mp1}
M(p,\{1\})
:=\frac{2}{p-1}\sum_{\chi\in X_p^-}\vert L(1,\chi)\vert^2
=\frac{\pi^2}{6}\left (1-\frac{1}{p}\right )\left (1-\frac{2}{p}\right )
\leq\frac{\pi^2}{6}
\ \ \ \ \ \hbox{($p\geq 3$)}
\end{equation}
and obtain the following upper bound
\begin{equation}\label{boundhpminusQzetap}
h_{{\mathbb Q}(\zeta_p)}^-
\leq 2p\left (\frac{pM(p,\{1\})}{4\pi^2}\right )^{(p-1)/4}
\leq 2p\left (\frac{p}{24}\right )^{(p-1)/4}.
\end{equation}
The mean square value of $L(1,\chi)$, $\chi\in X_p^-$ being asymptotic to $\pi^2/6$, 
by \eqref{Mp1},
for $K\subsetneq {\mathbb Q}(\zeta_p)$ we might expect to have bounds close to 
\begin{equation}\label{expected}
M(p,H)\leq\pi^2/6
\hbox{ and }
h_K^-\leq 2(p/24)^{n/2},
\end{equation}
by \eqref{hrelKbis}.
At least, according to Corollary \ref{MpHn} and Conjecture \ref{SpHnmod2} these bounds should hold true 
with probability greater than or equal to $1/2$.
However, it is hopeless to expect such a universal mean square upper bound. 
Indeed (e.g. see \cite{CK}), 
it is likely that there are infinitely many imaginary abelian number fields of a given degree $m=2n$ 
and prime conductors $p$ for which 
$$M(p,H) =\frac{1}{n}\sum_{\chi\in X_p^-(H)}\vert L(1,\chi)\vert^2
\geq\Bigl ( \prod_{\chi\in X_p^-(H)}L(1,\chi)\Bigr )^{2/n}
\gg (\log\log p)^{2}.$$

\section{An explicit formula for some mean square values of $L(1,\chi)$ over subgroups}
Formula \eqref{Mf1} gives an explicit formula for $M(f,\{1\})$ for $f>2$.
We now present in Theorem \ref{M(f,H)} 
the only situation where we could get an explicit formula for $M(f,H)$ 
for non trivial subgroups $H$ of the multiplicative group $({\mathbb Z}/f{\mathbb Z})^*$ 
where $f$ is not necessarily prime.

\begin{lemma}\label{aba'b'}
Let $f=\prod_{k=1}^tp_k^{e_k}$ be an integer 
such that all its $t\geq 1$ distinct prime divisors $p_k$ are equal to $1$ modulo $3$.
Then the set 
$$E_f
:=\{a/b\in ({\mathbb Z}/f{\mathbb Z})^*;\ f=a^2+ab+b^2\hbox{ and } \gcd (a,b)=1\}$$
is of cardinal $2^t$ 
and its elements are of order $3$ in the multiplicative group $({\mathbb Z}/f{\mathbb Z})^*$.\\ 
Moreover, if $\delta\geq 1$ divides $f$, 
then there exist $a'$ and $b'$ such that $\delta=a'^2+a'b'+b'^2$, 
$\gcd(a',b')=1$ 
and $a/b=a'/b'$ in $({\mathbb Z}/\delta{\mathbb Z})^*$. 
\end{lemma}

\begin{proof}
Let $p\equiv 1\pmod 3$ be prime. 
Then $p=\pi\bar\pi$ splits in the principal ideal domain ${\mathbb Z}[\zeta_6]$ 
of ${\mathbb Z}$-basis $\{1,\zeta_6\}$, where $\zeta_6$ is a primitive complex sixth root of unity 
and $\pi\in {\mathbb Z}[\zeta_6]$ is irreducible in ${\mathbb Z}[\zeta_6]$.
For each $k$, 
fix a factorisation of $p_k$ 
into a product of $2$ complex conjugates irreducible elements of ${\mathbb Z}[\zeta_6]$.
By taking 
$\alpha 
=a+b\zeta_6
=\prod_{k=1}^t\pi_k^{e_k}$, 
where $\pi_k$ is any of the $2$ given complex conjugate irreducible factors of $p_k$, 
we get $2^t$ ways to write $f =\alpha\bar\alpha =a^2+ab+b^2$ with $\gcd(a,b)=1$. 
Moreover, given two distinct such $\alpha$'s, 
say $\alpha_1 =a_1+b_1\zeta_6$ and $\alpha_2=a_2+b_2\zeta_6$, 
there exists some index $k$ for which $\pi_k$ divides $\alpha_1$ but does not divide $\alpha_2$. 
Since $b_2\alpha_1-b_1\alpha_2=a_1b_2-a_2b_1$ and $\gcd(b_1,f)=1$, 
it follows that $\pi_k$ does not divide $b_2\alpha_1-b_1\alpha_2$, 
which implies that $p_k$ does not divide $a_1b_2-a_2b_1$. 
Hence, $a_1/b_1\neq a_2/b_2$ in $({\mathbb Z}/f{\mathbb Z})^*$. 
Conversely, if $f=a^2+ab+b^2$ with $\gcd(a,b)=1$ 
then $f =\alpha\bar\alpha$ where $\alpha =a+b\zeta_6$ with $\gcd(a,b)=1$. 
Therefore, $\alpha =\eta\alpha_f$ for one of the $6$ invertible elements 
$\eta\in\{\zeta_6^k;\ 0\leq k\leq 5\}$ of ${\mathbb Z}[\zeta_6]$, 
where $\alpha_f:=\prod_{k=1}^t\pi_k^{e_k} =a_f+b_f\zeta_6$ 
is also such that $f=a_f^2+a_fb_f+b_f^2$. 
However, $a/b=a_f/b_f$ in $({\mathbb Z}/p{\mathbb Z})^*$. 
Indeed, 
$\zeta_6\alpha_f =-b_f+(a_f+b_f)\zeta_6$ and $-b_f/(a_f+b_f)=a _f/b_f$ in $({\mathbb Z}/p{\mathbb Z})^*$.\\ 
Finally, let $\delta =\prod_{k=1}^tp_k^{e_k'}$ be a divisor of $f$. 
Set $\beta:=\prod_{k=1}^t\pi_k^{e_k'} =a'+b'\zeta_6$, 
that divides $\alpha$ in ${\mathbb Z}[\zeta_6]$, say $\alpha =\beta\gamma$ with $\gamma =a''+b''\zeta_6$. 
Therefore, 
$a+b\zeta_6
=(a'+b'\zeta_6)(a''+b''\zeta_6) 
=(a'a''-b'b'')+(a'b''+a''b'+b'b'')\zeta_3$, 
hence $a=a'a''-b'b''$, $b=a'b''+a''b'+b'b''$, $\gcd(a',b')=1$
and noticing that $\delta =\beta\bar\beta =a'^2+a'b'+b'^2$ we obtain
$$a/b
=(a'a''-b'b'')/(a'b''+a''b'+b'b'')
=a'/b'$$ 
in $({\mathbb Z}/\delta{\mathbb Z})^*$.
\end{proof}

We can use $E_f$ to explicitly construct $2^{t-1}$ subgroups $\{1,a/b,b/a\}$ of order $3$ 
of the multiplicative group $({\mathbb Z}/f{\mathbb Z})^*$. 
Notice that by the Chinese remainder theorem 
there are $(3^t-1)/2$ subgroups of order $3$ in the multiplicative group $({\mathbb Z}/f{\mathbb Z})^*$
We will now prove the following new result:

\noindent\frame{\vbox{
\begin{theorem}\label{M(f,H)}
Let $f>1$ be an integer such that all its $t\geq 1$ distinct prime divisors are equal to $1$ modulo $3$. 
Then, for the $2^{t-1}$ subgroups $H_3=\{1,a/b,b/a\}$ 
of the multiplicative group $({\mathbb Z}/f{\mathbb Z})^*$ 
generated by the $2^t$ elements $a/b\in E_f$ 
we have
$$\tilde S(f,H_3)
:=\sum_{h\in H_3}\tilde s(h,f)
=\frac{\phi(f)}{12}\left (\prod_{p\mid f}\left (1+\frac{1}{p}\right )-\frac{1}{f}\right )$$ 
and (with the notation in \eqref{firstformula}, and compare with \eqref{Mf1})
$$M(f,H_3)
=\frac{\pi^2}{6}\frac{\phi(f)}{f}\left (\prod_{p\mid f}\left (1+\frac{1}{p}\right )-\frac{1}{f}\right ).$$
\end{theorem}
}}

\begin{proof}
We have 
$\tilde S(f,H_3)
=\tilde s(1,f)+2\tilde s(a/b,f)$ 
and we use \eqref{stilde(1,d)} and Lemma \ref{valuesab}.
\end{proof}

\begin{lemma}\label{valuesab}
Assume that $f =a^2+ab+b^2>3$, where $a,b\in {\mathbb Z}$ and $\gcd (a,b) =1$ 
(i.e. assume that all the prime divisors of $f$ are equal to $1$ mod $3$).
Set $h_3=a/b$, of order $3$ in the multiplicative group $({\mathbb Z}/f{\mathbb Z})^*$. 
Then for any divisor $\delta\geq 1$ of $f$ we have $s(h_3,\delta)=\frac{\delta-1}{12\delta}$. 
By \eqref{tildescddc}, it follows that 
$$\tilde s(h_3,f)
=\sum_{\delta\mid f}\frac{\mu(\delta)}{\delta}s(h_3,f/\delta)
=\sum_{\delta\mid f}\mu(\delta)\frac{f/\delta-1}{12f}
=\frac{\phi(f)}{12f}.$$
In particular, $\tilde s(h_3,f)$ does not depend on the choice of $h_3$ in $E_f$.
\end{lemma} 

\begin{proof}
By Lemma \ref{aba'b'} we have 
$s(h_3,\delta) =s(h_3',\delta)$ 
where $h_3'=a'/b'$ with $\delta =a'^2+a'b'+b'^2$ 
and $\gcd(a',b')=1$. 
By \cite[Lemma 4]{LouBPASM64} (see also \cite[Lemma 6.1]{LMCJM}), 
we have $s(h_3',\delta) =\frac{\delta-1}{12\delta}$.
\end{proof}

\begin{Remark}
(i). Let $f$ be a product of $t>1$ prime numbers equal to $1$ modulo $3$. 
There are $3^t-1$ elements $h$ of order $3$ 
in the multiplicative group $({\mathbb Z}/f{\mathbb Z})^*$ 
and $\tilde s(h,f)$ may depend on the choice of $h$.
For example, take $f =91 =7\cdot 13$. 
Then $h=29$ and $h'=53$ are of order $3$ in the multiplicative group $({\mathbb Z}/f{\mathbb Z})^*$ 
but such that $-\frac{22}{91}=\tilde s(h,f)\neq \tilde s(h',f)=-\frac{46}{91}$ 
are not equal and both different from $\frac{\phi(f)}{12f} =\frac{6}{91}$. 
Moreover,
$\frac{610}{91}=\tilde S(f,\{1,h,h^2\}) \neq S(f,\{1,h',h'^2\})=\frac{562}{91}$ 
are not equal and both not given by the formula in Theorem \ref{M(f,H)} 
that gives $\tilde S(f,H_3)=\frac{666}{91}$.\\
(ii). It seems difficult to find other situations 
where results similar to those in Theorem \ref{M(f,H)} would hold true. 
For example, take the moduli of the form $f=(a^5-1)/(a-1)$ 
with $\vert a\vert\geq 2$ and $a\not\equiv 1\pmod 5$.
Then the prime divisors of $f$ are equal to $1$ modulo $5$ 
and $H_5=\{1,a,a^2,a^3,a^4\}$ is a subgroup of order $5$ of $({\mathbb Z}/f{\mathbb Z})^*$. 
An explicit formula for $M(f,H_5)$ is known in the case that $f$ is prime 
(see \cite[Theorem 5]{LouBPASM64}). 
But we have not been able to obtain an explicit formula for $M(f,H_5)$ for non prime moduli $f$.
\end{Remark}

In particular, we have \eqref{expected} for some non-cyclotomic numbers fields:

\begin{corollary}\label{thp3}
(See \cite[Theorem 1]{LouBPASM64} or \cite[Theorem 6.6]{LMCJM}).
Let $p\equiv 1\pmod 6$ be a prime integer. 
Let $K$ be the imaginary subfield of degree $(p-1)/3$ of the cyclotomic number field ${\mathbb Q}(\zeta_p)$.
Let $H_3$ be the only subgroup of order $3$ 
of the multiplicative cyclic group $({\mathbb Z}/p{\mathbb Z})^*$. 
Then 
(compare with \eqref{Mp1})
$$M(p,H_3)
:=\frac{6}{p-1}\sum_{\chi\in X_p^-(H_3)}\vert L(1,\chi)\vert^2
=\frac{\pi^2}{6}\left (1-\frac{1}{p}\right )
\leq\frac{\pi^2}{6}.$$
Hence by \eqref{hrelKbis} we have (compare with \eqref{boundhpminusQzetap})
\begin{equation}\label{boundhpminusK}
h_K^-
\leq 2\left (\frac{pM(p,H_3)}{4\pi^2}\right )^\frac{p-1}{12}
\leq 2\left (\frac{p}{24}\right )^{(p-1)/12}
\end{equation}
(note the misprint in the exponent in \cite[(8)]{LouBPASM64}), 
i.e. the expected bounds \eqref{expected} hold true.
\end{corollary}

\section{On the denominator of Dedekind sums}
\begin{proposition}\label{ThRG}
(See \cite[Theorem 2 page 27]{RG}). 
For $\gcd (c,d)=1$ we have $2d\gcd (3,d)s(c,d)\in {\mathbb Z}$. 
\end{proposition}

If $p\equiv 7\pmod {12}$ then $2p\gcd (3,p)s(1,p) =(p-1)(p-2)/6$ is an odd integer coprime with $p$ 
and the information on the denominator of the rational number $s(c,d)$ given in Proposition \ref{ThRG} 
is optimal in this case. 
Hence, $2f\gcd (3,f)S(H_n,f)\in {\mathbb Z}$, 
where $H_n$ is a subgroup of order $n$ of the multiplicative group $({\mathbb Z}/f{\mathbb Z})^*$ 
and 
$$S(H_n,f)
:=\sum_{h\in H_n}s(h,f)
\in {\mathbb Q}.$$
We always have some cancelation on the common denominator $2f\gcd (3,f)$ of the $s(h,f)$'s 
when we sum over all the elements $h$ of a subgroup of order $n>1$ 
of the multiplicative group $({\mathbb Z}/f{\mathbb Z})^*$: 

\begin{theorem}\label{prediction}
Let $H_n$ be a subgroup of order $n$ of the multiplicative group $({\mathbb Z}/f{\mathbb Z})^*$. 
Set 
$$T(H_n,f):=\sum_{h\in H_n} h
\in {\mathbb Z}/f{\mathbb Z}.$$
(i). (See \cite[Lemma 5]{LouBKMS56}). If $n>1$, then $\gcd (f,T(H_n,f))>1$.\\
(ii). (See \cite[Theorem 10]{LouBKMS56}). If $f$ is odd then the rational number
$$2\gcd (3,f)\frac{f}{\gcd (f,T(H_n,f))}S(H_n,f)$$ 
is an integer of the same parity as $n\frac{f-1}{2}$.
\end{theorem}

Here is a general example which shows that the formulation of Theorem \ref{prediction} is optimal:

\noindent\frame{\vbox{
\begin{theorem}\label{SHp^nf}
Let $p\geq 3$ be prime, $f'>1$ odd and divisible by $p$ and $n\geq 1$. 
Set $f=p^nf'$ 
and
$$H_{p^n} =\{1+kf';\ 0\leq k\leq p^n-1\}
=\ker\left (({\mathbb Z}/f{\mathbb Z})^*\twoheadrightarrow ({\mathbb Z}/f'{\mathbb Z})^*\right ),$$ 
a subgroup of order $p^n$ of the multiplicative group $({\mathbb Z}/f{\mathbb Z})^*$.\\
Then 
$T(H_{p^n},f) =p^n+\frac{p^n-1}{2}f$, 
hence $\gcd (f,T(H_{p^n},f))$ $=p^n$, 
and
\begin{equation}\label{sums(1+kf',f)}
S(H_{p^n},f)
:=\sum_{h\in H_n}s(h,f)
=\sum_{k=0}^{p^n-1}s(1+kf',f)
=\frac{p^{n+1}+p^n-1}{12p^{n+1}}f
-\frac{p^n}{4}
+\frac{p^n}{6f}.
\end{equation}
Consequently (compare with Proposition \ref{ThRG}),
the rational number $2\gcd (3,f)\frac{f}{p^n}S(H_{p^n},f)$ 
is a rational integer not divisible by $p$
and it is odd if and only if $f\equiv 3\pmod 4$. 
\end{theorem}
}}

\begin{proof}
Take $d\geq 2$. 
Set $\zeta_d =\exp(2 i\pi /d)$.
Taking the logarithmic derivative of 
$\prod_{k=0}^{d-1}(x-\zeta_d^k)
=x^d-1$ at $x=1/\lambda$ 
we obtain 
\begin{equation}\label{step4bis}
\sum_{k=0}^{d-1}\frac{1}{\zeta_d^k\lambda-1}
=\frac{d}{\lambda^d-1}
\hbox{ whenever $\lambda^d\neq 1$.}
\end{equation} 
Taking the logarithmic derivative of 
$\prod_{k=1}^{d-1}(x-\zeta_d^k)
=(x^d-1)/(x-1) 
=x^{d-1}+\cdots +x+1$ at $x=1$ 
we obtain 
\begin{equation}\label{step0bis}
\sum_{k=1}^{d-1}\frac{1}{\zeta^k-1}
=-\frac{d-1}{2}.
\end{equation} 
Noticing that $\cot x =i+2i/(\exp(2ix)-1)$ and using \eqref{defs} and \eqref{step0bis}, 
we have 
\begin{equation}\label{scdbis}
s(c,d)
=\frac{d-1}{4d}-\frac{1}{d}\sum_{a=1}^{d-1}\frac{1}{(\zeta_d^a-1)(\zeta_d^{ac}-1)}
\ \ \ \ \ \hbox{(for $\gcd(c,d)=1$).}
\end{equation}
In particular, 
\begin{equation}\label{step1bis}
s(1+kf',f)
=\frac{f-1}{4f}
-\frac{1}{f}\sum_{a=1}^{f-1}\frac{1}{(\zeta_f^a-1)(\zeta_f^a\zeta_{p^n}^{ak}-1)}
\ \ \ \ \ \hbox{(for $f=p^nf'$ and $p\mid f'$),}
\end{equation}
and for $c=1$ and in using \eqref{s1d} we obtain 
\begin{equation}\label{step2bis}
\sum_{a=1}^{d-1}\frac{1}{(\zeta_d^a-1)^2}
=-\frac{(d-1)(d-5)}{12}
\ \ \ \ \ \hbox{(for $d\geq 2$).}
\end{equation}
Using \eqref{step1bis} and \eqref{step2bis}, we deduce that 
\begin{align}
\sum_{k=0}^{p^n-1}s(1+kf',f)
&=p^n\frac{f-1}{4f}
-\frac{1}{f}\sum_{a=1}^{f-1}\sum_{k=0}^{p^n-1}\frac{1}{(\zeta_f^a-1)(\zeta_f^a\zeta_{p^n}^{ak}-1)}\nonumber\\
&=p^n\frac{f-1}{4f}
-\frac{p^n}{f}\sum_{\substack{a=1\\ p^n\mid a}}^{f-1}\frac{1}{(\zeta_f^a-1)^2}
-\frac{1}{f}\sum_{\substack{a=1\\ p^n\nmid a}}^{f-1}
\sum_{k=0}^{p^n-1}\frac{1}{(\zeta_f^a-1)(\zeta_f^a\zeta_{p^n}^{ak}-1)}\nonumber\\
&=p^n\frac{f-1}{4f}
+\frac{(f-p^n)(f-5p^n)}{12p^nf}
-\frac{1}{f}\sum_{l=0}^{n-1}
\sum_{\substack{a=1\\ p^l\parallel a}}^{f-1}
\sum_{k=0}^{p^n-1}
\frac{1}{(\zeta_f^a-1)(\zeta_f^a\zeta_{p^n}^{ak}-1)}.\label{step5bis}
\end{align}
Now, using \eqref{step4bis} we obtain
\begin{align}
\sum_{\substack{a=1\\ p^l\parallel a}}^{f-1}
\sum_{k=0}^{p^n-1}
\frac{1}{(\zeta_f^a-1)(\zeta_f^a\zeta_{p^n}^{ak}-1)}
&=\sum_{\substack{a=1\\ p\nmid a}}^{f/p^l-1}
\sum_{k=0}^{p^n-1}
\frac{1}{(\zeta_{f/p^l}^a-1)(\zeta_{f/p^l}^a\zeta_{p^{n-l}}^{ak}-1)}\nonumber\\
&=\sum_{\substack{a=1\\ p\nmid a}}^{f/p^l-1}
\sum_{k=0}^{p^{n-l}-1}
\frac{p^l}{(\zeta_{f/p^l}^a-1)(\zeta_{f/p^l}^a\zeta_{p^{n-l}}^{ak}-1)}\nonumber\\
&=\sum_{\substack{a=1\\ p\nmid a}}^{f/p^l-1}\label{step6bis}
\frac{p^n}{(\zeta_{f/p^l}^a-1)(\zeta_{f/p^n}^a-1)}.
\end{align}
Since $\{a;\ 1\leq a\leq f/p^l-1\hbox{ and }p\nmid a\}
=\{Af/p^n+B;\ 0\leq A\leq p^{n-l}-1\hbox{ and }1\leq B\leq f/p^n-1\}$
$\setminus\{Af/p^n+pB;\ 0\leq A\leq p^{n-l}-1\hbox{ and }1\leq B\leq f/p^{n+1}-1\}$,
we have
\begin{align}
\sum_{\substack{a=1\\ p\nmid a}}^{f/p^l-1}
\frac{1}{(\zeta_{f/p^l}^a-1)(\zeta_{f/p^n}^a-1)}
&=\sum_{B=1}^{f/p^n-1}\frac{1}{\zeta_{f/p^n}^B-1}
\sum_{A=0}^{p^{n-l}-1}\frac{1}{\zeta_{p^{n-l}}^A\zeta_{f/p^l}^B-1}
\nonumber\\
&-\sum_{B=1}^{f/p^{n+1}-1}\frac{1}{\zeta_{f/p^{n+1}}^B-1}
\sum_{A=0}^{p^{n-l}-1}\frac{1}{\zeta_{p^{n-l}}^A\zeta_{f/p^{l+1}}^B-1}
\nonumber\\
&=\sum_{B=1}^{f/p^n-1}\frac{p^{n-l}}{(\zeta_{f/p^n}^B-1)^2}
-\sum_{B=1}^{f/p^{n+1}-1}\frac{p^{n-l}}{(\zeta_{f/p^{n+1}}^B-1)^2}\nonumber\\
&=-p^{-l}\frac{(f-p^n)(f-5p^n)}{12p^n}
+p^{-l}\frac{(f-p^{n+1})(f-5p^{n+1})}{12p^{n+2}},\label{step7bis}
\end{align}
by \eqref{step4bis} and \eqref{step2bis}. 
Finally, using \eqref{step5bis}, \eqref{step6bis} and \eqref{step7bis} we obtain
$$S(H_{p^n},f)
=p^n\frac{f-1}{4f}
+\frac{(f-p^n)(f-5p^n)}{12p^nf}
+\frac{1-p^{-n}}{1-p^{-1}}\frac{(f-p^n)(f-5p^n)}{12f}
-\frac{1-p^{-n}}{1-p^{-1}}\frac{(f-p^{n+1})(f-5p^{n+1})}{12p^{2}f}$$
and the desired formula.
\end{proof}

\begin{corollary}\label{elementsofsmallorder}
Let $p\geq 3$ be prime and $f=p^m$ with $m\geq 2$. 
Asssume that $1\leq n\leq m-1$.
Then $E_{p^n}=\{1+kf/p^n;\ 1\leq k\leq p^n-1\hbox{ and }\gcd (p,k)=1\}$ 
is the set of the $\phi(p^n) =p^{n-1}(p-1)$ elements of order $p^n$ of the multiplicative cyclic group 
$({\mathbb Z}/p^m{\mathbb Z})^*$
and we have the following mean value formula
$$\frac{1}{\# E_{p^n}}\sum_{h\in E_{p^n}}s(h,f)
=\frac{f}{12p^{2n}}-\frac{1}{4}+\frac{1}{6f}
\hbox{ for } 1\leq n\leq m-1.$$
\end{corollary}

\begin{proof}
We have 
$$\frac{1}{\# E_{p^n}}\sum_{h\in E_{p^n}}s(h,f)
=\frac{1}{p^{n-1}(p-1)}\left (
\sum_{k=1}^{p^n-1}s(1+kf/p^n,f)
-\sum_{k=1}^{p^{n-1}-1}s(1+kf/p^{n-1},f)
\right ).$$
Using \eqref{sums(1+kf',f)} for $n$ and $n-1$ the desired result follows.
\end{proof}

\begin{Remark}
Assume that $1\leq n\leq m-1$.
A Dirichlet character $\chi$ modulo $f =p^m$ is trivial on the subgroup $H_{p^n}$ 
if and only if it is induced by a Dirichlet character $\chi'$ modulo $f' =f/p^n$
and in that situation we have $L(1,\chi) =L(1,\chi')$. 
Therefore, noticing that $p^n/\phi (f) =1/\phi(f')$, 
we have 
$M(f,H_{p^n}) =M(f',\{1\})$. 
Now, on the one hand by \eqref{Mf1} we have 
$$M(f',\{1\})
=\frac{\pi^2}{6}\left (1-\frac{1}{p}\right )\left (1+\frac{1}{p}-\frac{3}{f'}\right ).$$
On the other hand by \eqref{firstformula} and \eqref{tildescddc} we have
$$M(f,H_{p^n})
=\frac{2\pi^2}{f}
\left (\sum_{h\in H_{p^n}}s(h,f)-\frac{1}{p}\sum_{h\in H_{p^n}}s(h,f/p)\right )
=\frac{2\pi^2}{f}
\left (S(H_{p^n},f)-S(H_{p^{n-1}},f/p)\right ).$$
Using Theorem \ref{SHp^nf}, we do recover that $M(f,H_{p^n}) =M(f',\{1\})$.
\end{Remark}

While checking the statements of Theorem \ref{SHp^nf} and Corollary \ref{elementsofsmallorder} 
on various values of $p$, $f'$, $n$ and $m$, 
we came across the following surprising Theorem 
which in the range $1\leq n\leq m/2$ is much more precise than Corollary \ref{elementsofsmallorder} 
and implies Corollary \ref{elementsofsmallorder}:

\noindent\frame{\vbox{
\begin{theorem}
Assume that $f\geq 1$ divides $f'^2$ and that $f'$ divides $f$. 
Then for $k\in {\mathbb Z}$ we have $\gcd(1+kf',f)=1$ and
\begin{equation}\label{propositions(1+kf',f)}
s(1+kf',f)
=\frac{f'^2}{12f}-\frac{1}{4}+\frac{1}{6f}
\hbox{ for $\gcd (k,f)=1$}.
\end{equation}
In particular, for $p\geq 3$ is prime, $f=p^m$ and $1\leq n\leq m/2$,
we have 
\begin{equation}\label{s(1+kf/p^n,f)}
s(1+kf/p^{n},f)
=\frac{f}{12p^{2n}}-\frac{1}{4}+\frac{1}{6f}
\hbox{ for $1\leq k\leq p^n-1$ and $\gcd(k,p)=1$.}
\end{equation} 
Therefore, for $m\geq 2$ and $1\leq n\leq m/2$, 
the Dedekind sums $s(h,p^m)$ 
do not depend on $h$ 
as $h$ runs over the $\phi(p^n) =p^{n-1}(p-1)$ elements of order $p^n$ 
of the multiplicative cyclic group $({\mathbb Z}/p^m{\mathbb Z})^*$.
\end{theorem}
}}

\begin{proof}
Set $q=f/f'$. Notice that $q$ divides $f'$.
By \eqref{scdbis}, proving \eqref{propositions(1+kf',f)} is equivalent to proving that:
$$\sum_{a=1}^{f-1}\frac{1}{(\zeta_{f}^a-1)(\zeta_{f}^a\zeta_{q}^{ak}-1)}
=\frac{6f-f'^2-5}{12}.$$
Write $a=Aq+B$ 
with $0\leq A\leq f'-1$, $0\leq B\leq q-1$ and $(A,B)\neq (0,0)$.
Then $(\zeta_{f}^a-1)(\zeta_{f}^a\zeta_{q}^{ak}-1)$
$=(\lambda\zeta_{f'}^A-1)(\mu\zeta_{f'}^A-1)$, 
where 
$\lambda 
=\lambda_B
=\zeta_{f}^{B}$
and 
$\mu
=\mu_{B,k}
=\zeta_{f}^B\zeta_{q}^{Bk}$. 
Since $\gcd (k,f)=1$ and $1\leq B\leq f'-1$, we have $\lambda\neq\mu$ and 
$$\frac{1}{(\lambda\zeta_{f'}^A-1)(\mu\zeta_{f'}^A-1)}
=-\frac{1}{\lambda-\mu}\left (\frac{\lambda}{\lambda\zeta_{f'}^A-1}-\frac{\mu}{\mu\zeta_{f'}^A-1}\right ).$$ 
Noticing that $\lambda^{f'} =\mu^{f'} =\zeta_{q}^B$, as $q\mid f'$, 
and using \eqref{step4bis} we get 
$$\sum_{A=0}^{f'-1}\frac{1}{(\lambda_B\zeta_{f'}^A-1)(\mu_{B,k}\zeta_{f'}^A-1)}
=-\frac{f'}{\zeta_{q}^B-1}
\hbox{ for $1\leq B\leq f'-1$}.$$
Therefore, we do have 
\begin{multline*}
\sum_{a=1}^{f-1}\frac{1}{(\zeta_{f}^a-1)(\zeta_{f}^a\zeta_{q}^{ak}-1)}
=\sum_{A=1}^{f'-1}\frac{1}{(\zeta_{f'}^A-1)^2}
+\sum_{B=1}^{q-1}\sum_{A=0}^{f'-1}\frac{1}{(\lambda_B\zeta_{f'}^A-1)(\mu_{B,k}\zeta_{f'}^A-1)}\\
=-\frac{(f'-1)(f'-5)}{12}
-f'\sum_{B=1}^{q-1}\frac{1}{\zeta_{q}^B-1}
=-\frac{(f'-1)(f'-5)}{12}+\frac{f'(q-1)}{2}
=\frac{6f-f'^2-5}{12},
\end{multline*}
by \eqref{step2bis} and \eqref{step0bis}.
\end{proof}

\begin{Remark}
The restriction $1\leq n\leq m/2$ is of paramount importance: 
for $m/2<n\leq m-1$ the Dedekind sum $s(1+kp^{m-n},p^m)$ may depend on $k$ with $1\leq k\leq p^n-1$ 
and $\gcd(k,p)=1$.
\end{Remark}

{\small
\bibliography{central}

}
\end{document}